\documentclass[12pt]{amsart}
\usepackage[all]{xy}
\usepackage{graphicx}
\usepackage{tikz}  
\usepackage{tikz-cd}
\usepackage{amsmath}
\usepackage{amssymb}
\usepackage{amsfonts}
\usepackage{mathrsfs}
\usepackage{mathtools}

\newcommand{\Cdb}{\ensuremath{\mathbb{C}}}

\newcommand{\Edb}{\ensuremath{\mathbb{E}}}

\newcommand{\Zdb}{\ensuremath{\mathbb{Z}}}

\newcommand{\M}{\mbox{${\mathcal M}$}}

\renewcommand{\S}{\mbox{${\mathcal S}$}}

\newcommand{\U}{\mbox{${\mathcal U}$}}

\newcommand{\norm}[1]{\Vert#1\Vert}

\newtheorem{theorem}{Theorem}[section]
\newtheorem{lemma}[theorem]{Lemma}

\newtheorem{corollary}[theorem]{Corollary}
\newtheorem{proposition}[theorem]{Proposition}
\newtheorem{definition}[theorem]{Definition}
\theoremstyle{remark}
\newtheorem{remark}[theorem]{\bf Remark}
\theoremstyle{definition}

\numberwithin{equation}{section}

\begin{document}

\title[absolute dilation of Fourier multipliers]
{absolute dilation of Fourier multipliers}

\bigskip
\author[C. Le Merdy]{Christian Le Merdy}
\email{clemerdy@univ-fcomte.fr}
\address{Laboratoire de Math\'ematiques de Besan\c con, 
Universit\'e Marie et Louis Pasteur, 16 route de Gray,
25030 Besan\c{c}on Cedex, FRANCE}

\author[S. Zadeh]{Safoura Zadeh}
\email{jsafoora@gmail.com}
\address{School of Mathematics, University of Bristol, United Kingdom}

\date{\today}


\begin{abstract} 
Let \( \M \) be a von Neumann algebra equipped with a normal semifinite faithful (nsf) trace. We say that an operator \( T :\M\to \M\) is absolutely dilatable if there exist another von Neumann algebra  
\( M \) with an nsf trace, a unital normal trace preserving \(\ast\)-homomorphism \( J: \M \to M \), and a trace preserving \(\ast\)-automorphism \( U: M \to M \) such that  \(T^k = {\mathbb E}_J U^k J \quad \text{for all } k \geq 0, \)  where \( {\mathbb E}_J: M \to \M \) is the conditional expectation associated with \( J \).  For a discrete amenable group \( G \) and a function \( u:G\to\mathbb{C} \) inducing a unital completely  positive Fourier multiplier \( M_u: VN(G) \to VN(G) \), we establish the following transference theorem: the operator \( M_u \) admits an absolute dilation if and only if its associated Herz-Schur multiplier does.  From this result, we deduce 
a characterization of Fourier multipliers with an absolute dilation in this setting.  Building on the transference result, we construct the first known example of a unital completely positive Fourier multiplier that does not admit an absolute dilation. This example arises in the symmetric group \( \S_3 \), the smallest group where such a phenomenon occurs.  Moreover, we show that for every abelian group \( G \), every Fourier multiplier always admits an absolute dilation.
\end{abstract}

\maketitle

\noindent
{\it 2000 Mathematics Subject Classification :}  47A20,
46L51, 43A22.

\smallskip
\noindent
{\it Key words:} Fourier multipliers, dilations, von Neumann algebras

\vskip 1cm

\section{Introduction}\label{Intro}
Rota’s \textit{Alternierende Verfahren} theorem in classical probability theory explores the convergence of iterates of measure
preserving Markov operators. Specifically, if \( T \) is a measure preserving Markov operator on a probability space \( (\Omega, \mu) \), then for any \( p \geq 1 \) and any function \( f \in L^p(\Omega, \mu) \), the sequence \( T^n (T^*)^n f \) converges almost everywhere. The proof relies on a dilation theorem.

In the noncommutative setting, C. Anantharaman-Delaroche extended this result in \cite{AD} by replacing the probability space with a von Neumann algebra \( M \) equipped with a normal faithful state \( \tau \). In this framework, \( T \) is a unital completely positive map preserving \( \tau \) and intertwining with the modular automorphism group of \( \tau \). Such maps, known as $\tau$-Markov operators, play a crucial role in noncommutative probability and quantum dynamics.

To extend Rota’s dilation theorem to the noncommutative setting, Anantharaman-Delaro\-che introduced an extra condition on \( T \) and referred to such operators as \emph{factorizable}. This class of operators 
encompasses all measure preserving 
Markov maps on abelian von Neumann 
algebras \cite[Remark 6.3(a)]{AD}, trace preserving 
Markov maps on \( M_2(\mathbb{C}) \) \cite{K}, 
and Schur multipliers 
associated with positive semidefinite matrices 
in \( M_n(\mathbb{R}) \) whose diagonal entries 
are all 
equal to 1 \cite{R}.
 
 In \cite{AD}, Anantharaman-Delaroche asked 
 if 
 all $\tau$-Markov operators are factorizable. 
 U. Haagerup and M. Musat \cite[Section 3]{HM} answered this question in the negative by providing examples 
 of non-factorizable trace preserving 
 Markov operators on the 
 matrix algebra $M_n(\mathbb{C})$ for $n\ge3$. 
 These examples are noteworthy not only as counterexamples but also because they occur in the tracial setting, 
 where one might expect nicer structural properties. 
 Haagerup and Musat also established an equivalence between factorizability and a dilation notion introduced 
 by Kümmerer \cite[Definition 2.1.1]{k1}. 
 In \cite[Definition 2.1]{DL1}, 
 C. Duquet and the first named author 
 examined Kümmerer’s dilation property, 
 the so-called \emph{absolute dilation property}, in the context of a von Neumann algebras equipped 
 with a normal semifinite faithful trace. 
 They then characterized the 
bounded Schur multipliers that admit an absolute dilation.

A related concept is that of Fourier multipliers. Let $G$ be a locally compact group and let $u\colon G\to\mathbb{C}$ be a continuous function that induces a unital completely positive Fourier multiplier $M_u\colon VN(G)\to VN(G)$.
É. Ricard \cite[Corollary]{R} showed that every 
such Fourier multiplier induced by a real-valued function \( u: G \to \mathbb{R} \) on a discrete group \( G \) is
absolutely dilatable (see also \cite[Theorem 4.6]{A}). 
C. Duquet \cite[Theorem 5.16]{D0} extended this result to the setting of non-discrete unimodular groups (see also \cite[Theorem 4.1]{DLM2}).
\medskip

Given $u\colon G\to\mathbb{C}$ such that \(M_u\) is unital and completely positive, we can associate a function
$\phi_u\colon G\times G\to\mathbb{C}$ defined by
\begin{equation}\label{HSdef}
\phi_u(s,t)=u(st^{-1}),\quad s,t\in G.
\end{equation}
Then $\phi_u$ induces a unital (completely) positive Schur multiplier, called the Herz–Schur multiplier,
\[
T_{\phi_u}\colon B(L^2(G))\to B(L^2(G)).
\]

In the first part of this paper, we examine the connection between the absolute dilatability of Fourier multipliers and Herz-Schur multipliers. We show that for discrete groups, if a Fourier multiplier \( M_u \) admits an absolute dilation, then the associated Herz-Schur multiplier is absolutely dilatable. Furthermore, we show that when the discrete group \( G \) is amenable, the converse also holds. Using \cite[Theorem 1.1]{DL1}, this provides a characterization of absolutely dilatable Fourier multipliers on amenable discrete groups.

The next result of this paper, which partially builds on the aforementioned development, presents an example of a unital completely positive Fourier multiplier that does not admit an absolute dilation. This provides the first known example
of such multipliers and is constructed on the symmetric group \( \S_3 \), the smallest group where such examples can
exist. Indeed, in Section \ref{Sec 3}, we show that when \( G \) is abelian, every unital (completely) positive 
Fourier multiplier admits an absolute dilation.

\section{Preliminaries and background}\label{Prelim}

\subsection{Tensor products}\label{SS-TP}
We assume that  the reader is familiar with the basics of von Neumann algebras, completely positive maps, and traces, and refer to \cite{Tak} for  details.
Given any two von Neumann algebras $M$ and $N$, we denote their von Neumann algebra tensor product by $M\overline{\otimes} N$. If $\Gamma\colon N_1\to N_2$ 
is a normal completely positive map, then for any von Neumann algebra
$M$, the tensor extension  $I\otimes \Gamma\colon M\otimes N_1\to M\otimes N_2
\subset M\overline{\otimes} N_2$ 
uniquely extends to a normal completely positive map
on $M\overline{\otimes} N_1$, denoted by
$$
I\overline{\otimes} \Gamma\colon
M\overline{\otimes} N_1\longrightarrow M\overline{\otimes} N_2.
$$
See \cite{NT} for details.

Let $M$ be a von Neumann algebra. If $\tau$ is a normal, semifinite, faithful trace on $M$, we call $(M, \tau)$ a tracial von Neumann algebra. If, in addition, $\tau$ is a state, we refer to $(M, \tau)$ as a normalized tracial von Neumann algebra.

Let $H$ be a Hilbert space, and let $B(H)$ denote the von Neumann algebra of all bounded operators on $H$. We denote its usual trace by ${\rm tr}$. Given a tracial von Neumann algebra $(M, \tau)$, we equip $B(H) \overline{\otimes} M$ with the unique normal, semifinite, faithful trace that maps $a \otimes b$ to ${\rm tr}(a) \tau(b)$ for any $a \in B(H)^+ \setminus \{0\}$ and $b \in M^+ \setminus \{0\}$ (see \cite[Proposition V.2.14]{Tak}).

\subsection{Fourier multipliers and Schur multipliers}\label{SS-M}
Let $G$ be a locally compact unimodular group equipped with a fixed Haar measure $ds$. We denote the unit of $G$ by $e_G$. Let $\lambda\colon G \to B(L^2(G))$ be the left regular representation of $G$, given by  
\[
[\lambda(s)h](t) = h(s^{-1}t), \quad h \in L^2(G), \ s,t \in G.
\]  
We denote its group von Neumann algebra by $VN(G)$, defined as  
\[
VN(G) = \overline{\{\lambda(t) : t \in G\}}^{w^*} \subseteq B(L^2(G)),
\]  
where the closure is taken in the $w^*$-topology of $B(L^2(G))$.  
Next, for any $f \in L^1(G)$, we define  
\[
\lambda(f) = \int_G f(s)\lambda(s)\, ds,
\]  
where the operator integral is understood in the strong sense. The operators $\lambda(f)$ belong to $VN(G)$.

We let $\tau_G$ denote the Plancherel weight on $VN(G)$ and refer to \cite{Haag} for details (see also \cite[Section VII.3]{Tak2}). Since $G$ is assumed to be unimodular, $\tau_G$ is actually a normal semifinite faithful trace. Moreover,  
\begin{equation}\label{Plancherel}
\tau_G\bigl(\lambda(f)^*\lambda(f)\bigr) = 
\norm{f}_{L^2(G)}^2, \quad f \in L^1(G) \cap L^2(G),
\end{equation}  
and $\tau_G$ is the unique normal semifinite faithful trace on $VN(G)$ with this property.

Following \cite{DCH}, we say that a continuous
function $u\colon G\to\Cdb$ induces a bounded Fourier
multiplier on $VN(G)$ if there exists a normal operator
$M_u : VN(G)\to VN(G)$ such that 
$$
M_u(\lambda(t))=u(t)\lambda(t),\quad t\in G.
$$
In this case, $M_u$ is unique. It is plain that $M_u$ is unital if and only if $u(e_G)=1$. Moreover, by \cite[Proposition 4.2]{DCH}, a continuous function $u\colon G\to\Cdb$ induces a completely positive Fourier
multiplier on $VN(G)$ if and only if $u$ is 
positive definite.

Let $I$ be an index set, and let $\phi = \{\phi(i,j)\}_{(i,j) \in I^2}$ be a family of complex numbers. We say that $\phi$ induces a bounded Schur multiplier on $B(\ell^2_I)$ if there exists a bounded operator $T_\phi: B(\ell^2_I) \to B(\ell^2_I)$ such that 
\begin{equation}\label{Schur}
T_\phi\bigl([a_{ij}]\bigr) = [\phi(i,j) a_{ij}], \quad [a_{ij}]_{(i,j) \in I^2} \in B(\ell^2_I).
\end{equation}
It is clear that $T_\phi$ is unital if and only if $\phi(k,k) = 1$ for all $k \in I$. The next result follows from \cite[Theorem 3.7]{P}.

\begin{lemma}\label{Positive}
A family $\phi = \{\phi(i,j)\}_{(i,j) \in I^2}$ induces a positive Schur multiplier $T_\phi$ if and only if, for any finite family $(k_1, \ldots, k_n)$ of elements in $I$, the matrix $[\phi(k_i, k_j)]_{1 \leq i,j \leq n}$ is positive semi-definite. Moreover, $T_\phi$ is completely positive in this case.
\end{lemma}

Let $G$ be a discrete group and let $u : G \to \mathbb{C}$ be a function. As indicated in the Introduction, we can associate a function $\phi_u : G \times G \to \mathbb{C}$ by setting 
$$
\phi_u(s, t) = u(st^{-1}),\quad s,t\in G.
$$

\begin{lemma}\label{HS}
A function $u: G\to \Cdb$ induces a unital completely positive Fourier multiplier 
on $VN(G)$ if and only if
$\phi_u$ induces a unital (completely) positive
Schur multiplier on $B(\ell^2_G)$.
\end{lemma}

This result goes back to \cite{BF}. The correspondence between Fourier and Schur multipliers can be extended to non-discrete groups and completely bounded maps, as shown in \cite{BF2, J}. However, we will not require this broader context.

Schur multipliers of the form $\phi_u$ are referred to as Herz-Schur multipliers.

\subsection{Factorizable and absolutely 
dilatable maps}\label{SS-Fact}
Let $(M, \tau)$ be a tracial von Neumann algebra. For any $1 \leq p < \infty$, we define $L^p(M,\tau)$ as the noncommutative $L^p$-space associated with $(M, \tau)$ \cite{PX}. Specifically, it is the completion 
of the space $\{x \in M : \tau(|x|^p) < \infty\}$ with respect to the norm
$$
\norm{x}_p = \left( \tau(|x|^p) \right)^{\frac{1}{p}}.
$$
In this paper, we will 
mostly focus on the case $p = 1$.

Additionally, for any Hilbert space $H$, the 
noncommutative $L^p$-spaces associated with
$(B(H), \text{tr})$ are the Schatten spaces
$S^p(H)$.

For any tracial von Neumann algebra $(M,\tau)$, the Banach space $L^1(M,\tau)$ is isometrically isomorphic to the predual $M_*$ of $M$. More specifically, $\tau$ extends to a norm one functional on $L^1(M,\tau)$, and the duality pairing on $M \times L^1(M,\tau)$, defined by
$$
\langle x, y \rangle_{M, L^1(M)} 
= \tau(xy),
$$
induces an isometric isomorphism
\begin{equation}\label{Dual}
M \simeq L^1(M,\tau)^*.
\end{equation}

Let $(M_1,\tau_1)$ and $(M_2,\tau_2)$ be two tracial von Neumann algebras. We say that a positive
operator $J: M_1\to M_2$ is trace preserving 
if $\tau_2\circ T=\tau_1$ on $M_1^+$. 

Assume that $J$ is a unital trace preserving $*$-homomorphism. Then $J$ is injective (because $\tau_1$ is faithful). Moreover, $J$ extends to an isometry $ J_1 : L^1(M_1, \tau_1) \to L^1(M_2, \tau_2). $
This follows readily from the fact that \( \tau_2(\vert J(x)\vert) = \tau_2(J(\vert x\vert)) = \tau_1(\vert x\vert)\) for all \(x\in M_1\).
Next, using (\ref{Dual}) for both $M_1$ and $M_2$, we define
$$
{\mathbb E}_J = J_1^* : M_2 \longrightarrow M_1.
$$
It follows from \cite[Proposition V.2.36]{Tak} and its proof that, when $M_1$ is viewed as a sub-von Neumann algebra of $M_2$ via $J$, the operator ${\mathbb E}_J$ acts as the natural conditional expectation from $M_2$ onto $M_1$. For this reason, ${\mathbb E}_J$ is commonly referred to as the conditional expectation associated with $J$.

\begin{definition}\label{Def AD}
Let \((\mathcal{M},\tau_{\mathcal{M}})\)
be a tracial von Neumann algebra. 
We say that an operator \(T:\M\to \M\) is
absolutely dilatable or admits an absolute
dilation if there exist another tracial von 
Neumann algebra \((M,\tau_M)\), a
normal unital trace preserving 
\(*\)-homomorphism \(J:\M\to M\), and a trace  
preserving \(*\)-automorphism \(U:M\to M\) 
such that 
\[
T^k = {\mathbb E}_J U^k J,\quad k\geq 0.
\]
\end{definition}

It is plain that if \(T:\M\to \M\) is absolutely dilatable, then $T$ is normal, unital, completely positive and trace preserving.

\begin{remark}\label{Normalized}
 Let  \(T:(\mathcal{M},\tau_{\mathcal{M}})\to (\mathcal{M},\tau_{\mathcal{M}})\) be an 
absolutely dilatable map and assume that
$\tau_{\mathcal{M}}$ is a state. Let $M,J,U$ as in Definition \ref{Def AD}. Then 
$\tau_M(1) = \tau_M(J(1))=\tau_{\mathcal M}(1)=1$, hence $\tau_M$ is a state. Thus, any
absolutely dilatable map on a normalized tracial von Neumann algebra is dilated through a normalized tracial von Neumann algebra.
\end{remark}

We will consider absolute dilations only for Fourier or Schur multipliers. In
this regard it is useful to note that
any unital positive Fourier multiplier
$M_u : VN(G)\to VN(G)$ is automatically trace preserving, see \cite[Lemma 3.1, (3)]{DLM2}.
Likewise, any unital positive Schur
multiplier $T_\phi\colon B(\ell^2_I)
\to B(\ell^2_I)$ is automatically trace preserving. This is clear from (\ref{Schur})
and the fact that $\phi(k,k)=1$ for all $k\in I
$ when $T_\phi$ is unital.

From now on, we use the abbreviation \(\bf{ucp}\) for ``unital completely positive". According to the previous discussion, we will focus on ucp Fourier or Schur multipliers.

\cite[Theorem 1.1]{DL1} provides a characterization of absolutely dilatable Schur multipliers constructed on a \(\sigma\)-finite measure space. Below, we present this result in the more specific context of discrete sets
(which traces back to \cite[Proposition 2.8]{HM} in the finite-dimensional case).

\begin{theorem}\label{DL1}
Let $I$ be any index set and let 
\(\phi = \{\phi(i,j)\}_{(i,j) \in I^2}\) 
be a family of complex numbers. 
The following assertions are equivalent:
\begin{itemize}
\item[(i)] The Schur multiplier 
\(T_\phi: B(\ell^2_I) \to B(\ell^2_I)\) is 
bounded and admits an absolute dilation;
\item[(ii)] There exist a 
normalized tracial von Neumann algebra 
\((N, \sigma)\) and a family \((d_k)_{k \in I}\) of unitaries 
in \(N\) such that for every \(i, j \in I\), 
\(\phi(i,j) = \sigma(d_j^* d_i)\).
\end{itemize}
\end{theorem}

We also note that, on the one hand, any ucp Schur multiplier $T_\phi: B(\ell^2_I) \to B(\ell^2_I)$ is absolutely dilatable if $\phi$ is real-valued 
(see \cite[Theorem 4.2]{A}). On the other hand, 
not all ucp Schur multipliers admit an absolute 
dilation (see \cite[Example 3.2]{HM} and \cite[Remark 7.3]{DL1}).

Let $(\M,\tau_{\mathcal{M}})$ be a normalized tracial von Neumann algebra and let $T : \M \to \M$ be a normal ucp trace preserving map. We say that $T$ is factorizable if there exist another tracial von Neumann algebra \((M,\tau_M)\), and two normal unital trace preserving \(*\)-homomorphisms \(J, J' : \M \to M\), such that $T = {\mathbb E}_{J} \circ J'$. This definition goes back to \cite{AD}, 
where it is introduced in the more general context of von Neumann algebras equipped with a (non necessarily tracial) normal
faithful state (see also \cite[Remark 1.4, (a)]{HM}).

The equivalence result below follows 
from \cite[Theorem 4.4]{HM} (using Remark \ref{Normalized}).

\begin{lemma}\label{Equiv}
Let $T : \M\to \M$ be a normal ucp trace preserving map on a normalized tracial von
Neumann algebra. Then $T$ admits an absolute 
dilation if and only if $T$ is factorizable.
\end{lemma}

We conclude this subsection by recalling two important results.  First, if $(\M, \tau_{\mathcal{M}})$ 
is an abelian normalized tracial von Neumann algebra, 
then any normal ucp trace preserving map on $\M$ is factorizable (see \cite[Remark 6.3(a)]{AD}). Second, if $G$ is a discrete group, then any ucp Fourier multiplier $T_\phi: VN(G) \to VN(G)$ is absolutely dilatable whenever $u$ is real-valued (see \cite[Corollary]{R}).  

We refer to \cite{AD, K, k1, R} for more results on factorizable maps.

\subsection{Ultraproducts}\label{Ultra}
We recall the construction of ultraproducts in the framework of normalized tracial von Neumann algebras. This construction goes back to McDuff \cite{McD}; for details, we refer to \cite[Section 11.5]{Pi}.

Let $(M_i, \sigma_i)_{i \in I}$ be a family of normalized tracial von Neumann algebras, and assume that $I$ is a directed set. Let $\U$ be an ultrafilter on $I$ refining the filter associated with the order on $I$. For any $i \in I$, let $\Vert \cdot \Vert_i$ denote the norm on $M_i$. Let
$$
\ell^\infty_I \{ M_i \} = \bigl\{ (y_i)_{i \in I} : y_i \in M_i \ \text{for all}\ i \in I \ \text{and} \ (\Vert y_i \Vert_i)_i \ \text{is bounded} \bigr\}
$$
be the direct sum of the $M_i$. Equipped with the norm $\Vert (y_i)_i \Vert = \sup_i \Vert y_i \Vert_i$ and the obvious product and involution, this is a $C^*$-algebra. Let
$$
Q = \bigl\{ (y_i)_i \in \ell_I^\infty \{ M_i \} : \|y_i\|_{L^2(M_i, \sigma_i)} \to 0 \bigr\}.
$$
Then $Q$ is a closed two-sided ideal of $\ell_I^\infty \{ M_i \}$. We may therefore define the quotient $C^*$-algebra
$$
M = \frac{\ell_I^\infty \{ M_i \}}{Q}.
$$

In the sequel, we let ${\rm cl}\bigl((y_i)_i\bigr)\in M$
denote the class of $(y_i)_i$ modulo $Q$, for any $(y_i)_i\in \ell^\infty_I\{M_i\}$. 
It turns out that  $M$ is actually a von Neumann 
algebra and that the functional $\sigma : M\to\Cdb$ defined
by
$$
\sigma\bigl({\rm cl}\bigl((y_i)_i\bigr)\bigr) := \lim_{\mathcal U} \sigma_i(y_i),\quad (y_i)_i\in \ell^\infty_I\{M_i\},
$$
is a well-defined normal tracial faithful state on $M$.
The resulting normalized tracial von Neumann algebra $(M,\sigma)$ 
is called the ultraproduct of $\{(M_i,\sigma_i)\}_{i\in I}$ along
$\U$.

\section{Absolute dilation of Fourier multipliers on 
discrete groups}\label{Sec 3}
In this section, we apply Theorem \ref{DL1}, 
along with a transference method, to describe absolutely dilatable Fourier multipliers on
amenable discrete groups. To achieve this,
we adapt a method from \cite{NR} to the 
context of absolute dilations.

\begin{theorem}\label{Transf}
Let \(G\) be a discrete 
group, and let \(u: G \to \mathbb{C}\) be a unital, positive
definite function (so that the Fourier
multiplier $M_u : VN(G)\to VN(G)$ is ucp).
Then the following assertions hold:
\begin{itemize}
\item[(i)] If \(M_u\) admits an absolute dilation, then the associated Schur multiplier \(T_{\phi_u}\) is absolutely dilatable.
\item[(ii)] If \(G\) is amenable and if \(T_{\phi_u}\) is absolutely dilatable, then \(M_u\) admits an absolute dilation. 
\end{itemize}
\end{theorem}

\begin{proof} (i) 
Assume that the Fourier multiplier
\(M_u: VN(G) \to VN(G)\) admits an absolute dilation. 
Then there exist a (normalized) tracial von Neumann algebra \((M, \Psi)\), a normal unital trace preserving \(\ast\)-homomorphism \(J: VN(G) \to M\), and a trace preserving 
$\ast$-automorphism \(U: M \to M\) such that for all \(k \geq 0\),
\[
        M_u^k = J_1^* U^k J.
 \]

We now appeal to Subsection \ref{SS-TP}.
We regard elements of $B(\ell^2_G) \overline{\otimes} VN(G)$
as matrices $[x_{st}]_{(s,t)\in G^2}$ with entries $x_{st}\in VN(G)$ 
in the usual way.
We will use the inclusion \( \ell^\infty_G \subseteq B(\ell^2_G) \) 
given by the diagonal action.
Consider the unitary  \(V = (\lambda(t))_{t \in G} \in 
\ell^\infty_G(VN(G))\), 
regarded as an element of \(B(\ell^2_G) \overline{\otimes} VN(G)\), 
via the identification 
\( \ell^\infty_G(VN(G)) \simeq \ell^\infty_G \overline{\otimes} VN(G) \), and the aforementioned 
inclusion \( \ell^\infty_G \subseteq B(\ell^2_G) \).
Next, define \(\sigma: B(\ell^2_G) \to B(\ell^2_G) \overline{\otimes} VN(G)\) by 
\[
\sigma(a) = V(a \otimes 1) V^*, \quad  a \in B(\ell^2_G).
\]
Then, \(\sigma\) is a normal unital \(\ast\)-homomorphism. Since
the Plancherel trace on discrete groups is normalized,
$\sigma$ is trace preserving. 
Expanding its action on a matrix 
$a = [a_{st}]\in B(\ell^2_G)$,
we obtain
\begin{align*}
\sigma(a) &= \operatorname{Diag}(\lambda(s)) [a_{st} \cdot 1] \operatorname{Diag}(\lambda(t^{-1})) \\
&= \left[a_{st}  \lambda(s)\lambda(t^{-1})\right] \\
&= \left[a_{st} \lambda(st^{-1})\right].
\end{align*}
This implies the intertwining relation
$$
(I \overline{\otimes} M_u) \circ \sigma = \sigma \circ T_{\phi_u}.
$$
Iterating this equation, 
we deduce that for all \(k \geq 0\),
\[
(I \overline{\otimes} M_u)^k \circ \sigma = 
\sigma \circ T_{\phi_u}^k.
\]
Hence, for all \(k\geq0\),
\[
\sigma_1^* \circ (I \overline{\otimes} M_u)^k \circ \sigma = T_ {\phi_u}^k.
\]
Thus we have the following commutative diagram

\[
\begin{tikzcd}[column sep=5em, row sep=5em, shorten <= 2pt, shorten >= 2pt]
B(\ell^2_G) \overline{\otimes} M \arrow[r, "(I \overline{\otimes} U)^k"] & B(\ell^2_G) \overline{\otimes} M \arrow[d, "(I \overline{\otimes} J)^*_1"] \\
B(\ell^2_G) \overline{\otimes} VN(G) \arrow{u}[swap]{I \overline{\otimes} J}  \arrow[r, "(I \overline{\otimes} M_u)^k"] & B(\ell^2_G) \overline{\otimes} VN(G) \arrow[d, "\sigma_1^*"]\\
B(\ell^2_G) \arrow[r, "T_{\phi_u}^k"] \arrow[u, "\sigma"'] & B(\ell^2_G) 
\end{tikzcd}
\]

Hence, \(T_{\phi_u}\) is absolutely dilatable.

\smallskip
(ii) Assume that \( T_{\phi_u} \) is absolutely dilatable. 
Since \( G \) is amenable, it admits a Følner net, which is a net 
\( (F_i)_{i\in I} \) of finite subsets of \( G \) 
satisfying the property that for every \( t \in G \),  
\begin{equation}\label{eq2}
\frac{|F_i \triangle tF_i|}{|F_i|} \to 0.
\end{equation}  
Equivalently, for every \( t \in G \),  
\begin{equation}\label{eq3}
\frac{|F_i \cap tF_i|}{|F_i|} \to 1.
\end{equation}
Here, $|F|$ denotes the cardinality of any subset $F$ of $G$.

For any \(i\in I\), let \(T_i: B(\ell^2_{F_i}) \to B(\ell^2_{F_i})\) 
be the restriction of \(T_{\phi_u}\). Applying Theorem \ref{DL1}, 
first to 
$T_{\phi_u}$ and then to $T_i$, 
we obtain that \(T_i\) is absolutely dilatable.
Let \(\tau_i = \frac{1}{\lvert F_i \rvert} \operatorname{tr}(\cdot)\) 
be the normalized trace on \(B(\ell^2_{F_i})\). 
Then, there exist a normalized tracial von Neumann algebra \((M_i, \sigma_i)\), a 
normal unital \(\ast\)-homomorphism 
\(J_i: B(\ell^2_{F_i}) \to M_i\) such that 
\(\sigma_i \circ J_i = \tau_i\), and a 
trace preserving \(\ast\)-automomorphism \(U_i: M_i \to M_i\) 
such that for every \(k \geq 0\),
\[
T_i^k = \mathbb{E}_i U_i^k J_i,
\]
where \(\mathbb{E}_i: M_i \to B(\ell^2_{F_i})\) is the conditional expectation associated with \(J_i\). Further, 
let \(\Gamma_i: \ell^2_{F_i} \to \ell^2_G\) be the canonical embedding, and define \(\alpha_i: VN(G) \to B(\ell^2_{F_i})\) by  
\[
\alpha_i(x) = \Gamma_i^* x \Gamma_i,\quad x\in VN(G).
\]
We will work with the diagram below.
\[
\begin{tikzcd}[column sep=3em, row sep=3em, shorten <= 2pt, shorten >= 2pt]
(M_i, \sigma_i) \arrow[r, "U_i^k"] 
& (M_i, \sigma_i) \arrow[d, "\mathbb{E}_i"] \\
  (B(\ell^2_{F_i}), \tau_i) \arrow[u, "J_i"'] \arrow[r, "T_i^k"] & (B(\ell^2_{F_i}), \tau_i) \\
  VN(G) \arrow[r, "M_u^k"] \arrow[u, "\alpha_i"'] & VN(G).
\end{tikzcd}
\]

Let \(\mathcal{U}\) 
be an ultrafilter on I refining the filter 
associated with the order
on I and  let $(M,\sigma)$ be the ultraproduct 
of $\{(M_i,\sigma_i)\}_{i\in I}$ along
$\U$.
We refer to Subsection \ref{Ultra} for this construction and we use the notation from there.

Let \((y_i)_{i \in I} \in \ell_I^\infty\{M_i\}\). Since 
each \(U_i\) is a trace preserving \(\ast\)-automorphism, we have  
\[
\sigma_i(U_i(y_i)^* U_i(y_i)) = \sigma_i(U_i(y_i^* y_i)) = \sigma_i(y_i^* y_i).
\]
Thus, if \((y_i)_{i \in I} \in Q\), then \((U_i(y_i))_{i\in I}
\in Q\). Consequently, we may define \(U : M \to M\) by  
\[
U\bigl(\mathrm{cl}\bigl((y_i)_i\bigr)\bigr) = 
\mathrm{cl}\bigl((U_i(y_i))_i\bigr).
\]
It is straightforward to verify that \(U\) is a 
trace preserving \(\ast\)-automorphism.

Let $(\epsilon_t)_{t\in G}$ denote the standard
basis of $\ell^2_G$ and let 
$(E_{s,t})_{(s,t)\in G^2}$ denote 
the matrix units of $B(\ell^2_G)$ corresponding to 
this basis. Then for any \(s \in G\), we have
$\lambda(s)\epsilon_q=\epsilon_{sq}$ for all
$q\in G$, hence
\[
\lambda(s) = \sum_{q \in G} E_{sq, q},
\]
where the summation is taken in the strong operator topology.
Hence for any $i\in I$,
\[
\alpha_i(\lambda(s)) = \sum_{q \in F_i \cap s^{-1} F_i} E_{sq, q}.
\]
Thus, \(\tau_i(\alpha_i(\lambda(s))) = 0\) if \(s \neq e_G\), and \(\tau_i(\alpha_i(\lambda(e_G))) = 1\). Since
$\alpha_i$ is normal, we deduce that 
\begin{equation}\label{TP}
\tau_i \circ \alpha_i = \tau_G.
\end{equation}

For any \(x \in VN(G)\), the 
family \((J_i \alpha_i(x))_{i\in I}\) is bounded. 
Hence, we may define a linear map \(J: VN(G) \to M\) by
\[
J(x) = \mathrm{cl}\bigl((J_i \alpha_i(x))_i\bigr), \quad 
x \in VN(G).
\]
It is clear that $J$ is unital and positive. Moreover, 
$$
\sigma\circ J=\tau_G.
$$
Indeed, for all \(x \in VN(G)\),
$$
\sigma(J(x)) = \lim_{\mathcal{U}} \sigma_i(J_i \alpha_i(x)) = \lim_{\mathcal{U}} \tau_i(\alpha_i(x)) = \tau_G(x),
$$
by (\ref{TP}). According to \cite[Proposition 5.4]{HJX},
this implies that \(J\) is normal.

Let us show that \(J\) is a \(\ast\)-homomorphism. It is clear that \(J(x)^* = J(x^*)\) for all \(x \in VN(G)\). Let \(s,t \in G\) and consider \(i \in I\). First, we have
\[
\alpha_i(\lambda(s)) \alpha_i(\lambda(t)) = \Gamma_i^* \lambda(s) \Gamma_i \Gamma_i^* \lambda(t) \Gamma_i.
\]
We note that
$$ 
\lambda(t) = \sum_{q \in G} E_{tq, q}, \quad
\Gamma_i \Gamma_i^* = \sum_{p \in F_i} E_{p,p}
\quad\hbox{and}\quad
\lambda(s) = \sum_{r \in G} E_{r, s^{-1} r}.
$$
Hence,
\[
\Gamma_i \lambda(s) \Gamma_i \Gamma_i^* \lambda(t) \Gamma_i = \left( \sum_{r \in F_i} E_{r, s^{-1} r} \right) \left( \sum_{p \in F_i} E_{p,p} \right) \left( \sum_{q \in F_i} E_{tq, q} \right).
\]
Arguing as above, we obtain
\[
\alpha_i(\lambda(s)) \alpha_i(\lambda(t)) = \sum_q E_{stq, q},
\]
where the sum runs over all \(q \in F_i\) such that \(tq \in F_i\) and \(stq \in F_i\), that is, 
$$
q \in F_i \cap t^{-1} F_i \cap t^{-1} s^{-1} F_i.
$$
Second, we have \(\alpha_i(\lambda(s)\lambda(t))=\Gamma_i^*\lambda(st)\Gamma_i\). 
Since \(\lambda(st)=\sum_{q\in G}E_{stq,q}\), we obtain 
$$
\Gamma_i^*\lambda(st)\Gamma_i=\sum_q E_{stq,q},
$$
where the sum runs over all \(q\in F_i\) 
such that \(stq\in F_i\), that is 
$$
q\in F_i\cap t^{-1}s^{-1}F_i. 
$$
Let \(H_i=F_i\cap t^{-1}s^{-1}F_i\). 
Combining the above two calculations, we obtain 
\[
\alpha_i(\lambda(s)\lambda(t))-\alpha_i(\lambda(s))\alpha_i(\lambda(t))=
\sum_{H_i\setminus H_i\cap t^{-1}F_i}E_{stq,q}. 
\]
Hence,
\[
\bigl\Vert\alpha_i(\lambda(s)\lambda(t))-\alpha_i(\lambda(s))\alpha_i(\lambda(t))
\bigr\Vert^2_{L^2
(B(\ell^2_{F_i}))}=
\frac{\lvert H_i\setminus H_i\cap t^{-1}F_i\rvert}{\lvert F_i\rvert}.
\]
Since 
$$
H_i\setminus H_i\cap t^{-1}F_i\subseteq F_i\setminus F_i\cap t^{-1}F_i\subseteq F_i\triangle t^{-1} F_i,
$$
we deduce that 
\[
\bigl\Vert\alpha_i(\lambda(s)\lambda(t))-
\alpha_i(\lambda(s))\alpha_i(\lambda(t))\bigr\Vert^2_{L^2(B(\ell^2_{F_i}))}
\leq\frac{\lvert F_i\triangle t^{-1}F_i\rvert}{\lvert F_i\rvert}.
\]
Since \(\sigma_i\circ J_i =\tau_i\) and $J_i$ is
a \(\ast\)-homomorphism, it preserves the \(L^2\)-norm. 
Hence we actually have
\[
\bigl\Vert J_i\alpha_i(\lambda(s)\lambda(t))-
J_i\alpha_i(\lambda(s))J_i\alpha_i(\lambda(t))
\bigr\Vert_{L^2(M_i)}^2\leq
\frac{\lvert F_i\triangle t^{-1}F_i\rvert}{\lvert F_i\rvert}.
\]
The upper bound goes to zero as \(i\to\infty\), by (\ref{eq2}). 
Consequently, 
$$
J_i \alpha_i(\lambda(s)\lambda(t)) - J_i \alpha_i(\lambda(s)) J_i \alpha_i(\lambda(t)) \in Q,
$$
and hence, \(J(\lambda(s)\lambda(t)) = J(\lambda(s)) J(\lambda(t))\). 
By linearity and normality of \(J\), we deduce as expected
that for every \(x, y \in VN(G)\),
we have
$J(xy) = J(x) J(y)$.

We now conclude the proof by showing that for every \(k \geq 0\),
\[
M_u^k = \mathbb{E}_J U^k J.
\]
By linearity and normality again, it suffices 
to show that for every \(t \in G\),
\[
\mathbb{E}_J U^k J(\lambda(t)) = u(t)^k \lambda(t).
\]
We fix \(t \in G\). For all \(s \in G\),
\begin{align*}
\langle \lambda(t), \lambda(s) \rangle_{VN(G), L^1(VN(G))} 
&= \tau_G(\lambda(t) \lambda(s)) \\
&= \tau_G(\lambda(ts)) = \begin{cases}
        1 & \text{if } s = t^{-1}, \\
        0 & \text{otherwise}.
    \end{cases}
\end{align*}
By density of \(\lambda(G)\) in \(L^1(VN(G))\), it therefore suffices to show that
\begin{equation}\label{eq4}
\langle U^k J(\lambda(t)), J(\lambda(s)) \rangle_{M, L^1(M)} = \begin{cases}
        u(t)^k & \text{if } s = t^{-1}, \\
        0 & \text{otherwise}.
\end{cases}
\end{equation}

To see this, recall that for any \(i \in I\) 
and any \(a \in B(\ell^2_{F_i})\), 
we have \((\mathbb{E}_i U_i^k J_i)(a) = T_i^k(a)\). Hence
for  any \(a,b \in B(\ell^2_{F_i})\),
$$
\langle \mathbb{E}_i U_i^k J_i(a), b \rangle= 
\langle T_i^k(a), b \rangle,
$$
and hence,
\[
\langle U_i^k J_i(a), J_i(b) \rangle= \langle T_i^k(a), b \rangle.
\]
Therefore,
\[
\sigma_i\bigl(U_i^k J_i(a) \cdot J_i(b)) = \langle T_i^k(a), b \rangle.
\]
Let us apply this with \(a = \alpha_i(\lambda(t))\) and \(b = \alpha_i(\lambda(s))\). Since \(\alpha_i(\lambda(t)) = \sum_{q \in F_i \cap t^{-1} F_i} E_{tq, q}\), 
we have
\[
T_i^k\bigl(
\alpha_i(\lambda(t))\bigr) 
= \sum_{q \in F_i \cap t^{-1} F_i} u(t)^k E_{tq, q}.
\]
Consequently,
\begin{align*}
T_i^k\bigl(\alpha_i(\lambda(t))\bigr)
\alpha_i(\lambda(s)) &= u(t)^k 
\left( \sum_{q \in F_i \cap t^{-1} F_i} E_{tq, q} \right) 
\left( \sum_{r \in F_i 
\cap s^{-1} F_i} E_{sr, r} \right) \\
&= u(t)^k \sum_{r \in F_i \cap s^{-1} F_i \cap 
s^{-1} t^{-1} F_i} E_{tsr, r}.
\end{align*}
If \(s \neq t^{-1}\), then \(ts \neq e_G\). Hence, 
\({\rm tr}(E_{tsr, r}) = 0\), for all \(r\), so
$$
\bigl\langle T_i^k\bigl(\alpha_i(\lambda(t))\bigr), 
\alpha_i(\lambda(s)) \bigr\rangle
=\tau_i
\bigl(T_i^k\bigl(\alpha_i(\lambda(t))\bigr) 
\alpha_i(\lambda(s)))\bigr) = 0.
$$
Hence, 
$$
\sigma_i\bigl(U_i^k J_i \alpha_i(\lambda(t)) \cdot J_i \alpha_i(\lambda(s))\bigr) = 0.
$$
In the case \(s = t^{-1}\), we obtain
\[
T_i^k\bigl(\alpha_i(\lambda(t))\bigr) \alpha_i(\lambda(t^{-1})) = u(t)^k \sum_{r \in F_i \cap t F_i} E_{r,r},
\]
hence
\[
\tau_i\bigl(T_i^k\bigl(\alpha_i(\lambda(t))\bigr)
\alpha_i(\lambda(t^{-1}))\bigr) = u(t)^k \,
\frac{|F_i \cap t F_i|}{|F_i|},
\]
and hence 
$$
\sigma_i\bigl(U_i^k J_i 
\alpha_i(\lambda(t)) \cdot J_i 
\alpha_i(\lambda(t^{-1}))\bigr) = u(t)^k\, 
\frac{|F_i \cap t F_i|}{|F_i|}.
$$

For any \(t, s \in G\), we have
\begin{align*}
\langle U^k J(\lambda(t)), J(\lambda(s)) \rangle_{M, L^1(M)}
&= \sigma\left(U^k J(\lambda(t)) \cdot J(\lambda(s))\right) \\
&= \lim_{\mathcal{U}} \sigma_i\bigl(U_i^k J_i 
\alpha_i(\lambda(t)) \cdot J_i 
\alpha_i(\lambda(s))\bigr).
\end{align*}
By the preceding calculation, this is equal to zero if 
\(s \neq t^{-1}\), and
\[
\langle U^k J(\lambda(t)), J(\lambda(t^{-1})) \rangle_{M, L^1(M)} = u(t)^k \,\lim_{\mathcal{U}} \frac{|F_i \cap t F_i|}{|F_i|}.
\]
By \eqref{eq3}, the limit on the right-hand side is equal to 
$u(t)^k$. Hence, \eqref{eq4} is proved.\qedhere
\end{proof}

\begin{corollary}
Let $G$ be a discrete amenable group and let 
$u\colon G\to \Cdb$ be a unital, positive definite function.
The following assertions are
equivalent:
\begin{itemize}
\item [(i)] The Fourier multiplier $M_u : VN(G)\to VN(G)$ 
admits an absolute dilation;
\item [(ii)] The Fourier multiplier $M_u : VN(G)\to VN(G)$ is factorizable;
\item [(iii)] There exist a normalized tracial von Neumann algebra 
\((N, \sigma)\) and a family \((d_t)_{t\in G}\) of unitaries 
in \(N\) such that for every $s,t\in G$, $u(st^{-1})= 
\sigma(d_s^* d_t)$.
\end{itemize}
\end{corollary}

\begin{proof}
Combine Theorem \ref{Transf}, Theorem \ref{DL1} and Lemma \ref{Equiv}.
\end{proof}

We do not know if property (ii) in Theorem \ref{Transf} holds 
in the non amenable case. That is, we do not know 
if the two conditions 
``$M_u$ admits an absolute dilation" and ``$T_{\phi_u}$
admits an absolute dilation" are equivalent on any discrete group.

\section{Absolute dilation on non discrete groups}
A natural question is whether Theorem \ref{Transf} can be extended to non discrete groups. In this brief section, we provide two partial answers to this question: first for abelian groups, and second for compact groups.

\begin{theorem}\label{abgp}
Let \( G \) be an abelian locally compact 
group whose dual group $\widehat{G}$ is $\sigma$-compact.
Let $u: G\to \Cdb$ be a unital, positive definite continuous function.
Then the Fourier multiplier
\( M_u: VN(G) \to VN(G) \) admits an absolute dilation.
\end{theorem}

\begin{proof}
Let $dt$ denote the Haar measure on $\widehat{G}$. We will use the
fact that this measure is $\sigma$-finite (see e.g. \cite[Lemma 3.2]{AKLZ}). We recall the classical identification
$$
VN(G)\simeq L^\infty\bigl(\widehat{G}\bigr),
$$
in which $\tau_G$ corresponds to the integral with respect to $dt$.
Let $T\colon L^\infty\bigl(\widehat{G}\bigr)\to L^\infty\bigl(\widehat{G}\bigr)$ corresponding to $M_u$
in this identification.
Since $T$ is unital and positive,
there exists a probability measure $\mu$ on 
$\widehat{G}$ such that $T=\mu\ast\,\cdotp\,$ (=
convolution by $\mu$).

Consider the (infinite) product $\widehat{G}^{\,\mathbb Z}$, 
and let $\nu$ be the probability measure on that space
defined as  the infinite tensor product of $\mu$ over $\Zdb$.
Then we set
$$
M:=L^\infty\Bigl(\widehat{G}\times\widehat{G}^{\,\mathbb Z}, dt\otimes d\nu(s)\Bigr),
$$
and we equip $M$ 
with the (normal, faithful, semifinite) trace
given by the integral with respect to 
$dt\otimes d\nu(s)$. Here we use that $dt$ is $\sigma$-finite.

We let $\sigma\colon M\to M$ be defined by
$$
\bigl[\sigma(F)\bigr]\bigl(t,(s_k)_{k\in{\mathbb Z}}\bigr) 
=F\bigl(t, (s_{k+1})_{k\in{\mathbb Z}}\bigr),
\quad F\in M.
$$
Clearly, $\sigma$ is a unital trace preserving $\ast$-automorphism.

Next, let $\pi\colon M\to M$ be defined as follows. For any 
$h\in M$,
$$
[\pi(h)](t,s) = h(t -s_0,s),\qquad t\in\widehat{G},\ s=(s_k)_{k\in{\mathbb Z}}\in \widehat{G}^{\,\mathbb Z}.
$$
Then $\pi$ is a $\ast$-homomorphism 
and it readily follows from the definition of the 
Haar measure that $\pi$
is trace preserving. We set 
$$
U=\pi\circ \sigma.
$$
Then $U$ is a trace preserving $*$-automorphism of $M$.

We let $J\colon L^\infty\bigl(\widehat{G}\bigr)\to M$ be defined by 
$$
J(f)=f\otimes 1 : \bigl(t, (s_{k+1})_{k\in{\mathbb Z}}\bigr)
\mapsto f(t),\quad f\in L^\infty\bigl(\widehat{G}\bigr).
$$
This is a normal unital and
trace preserving $*$-homomorphism. Let $\Edb_J\colon M\to  L^\infty\bigl(\widehat{G}\bigr)$ be the conditional expectation 
associated with $J$. 
It is plain that we have
\begin{equation}\label{E}
[\Edb_J(h)](t)= \int_{\widehat{G}^{\,\mathbb Z}} h(t,s)\, d\nu(s),\qquad h\in
L^\infty\bigl(\widehat{G}\times\widehat{G}^{\,\mathbb Z}, dt\otimes d\nu(s)\bigr),\ t\in\widehat{G}.
\end{equation}

Now let $f\in L^\infty\bigl(\widehat{G}\bigr)$. Then 
$[UJ(f)](t,s)=f(t-s_0)$ for $t\in \widehat{G}$ and $s=(s_k)_k\in \widehat{G}^{\,\mathbb Z}$. Hence
$$
[\sigma UJ(f)](t,s) = f(t-s_1),\qquad 
t\in \widehat{G}, \ s\in \widehat{G}^{\,\mathbb Z},
$$
and hence 
$$
[U^2J(f)](t,s)=f(t-s_0-s_1),\qquad 
t\in \widehat{G}, \ s\in \widehat{G}^{\,\mathbb Z}.
$$
By induction we obtain that for any integer $k\geq 0$,
$$
[U^kJ(f)](t,s)=f(t-s_0-s_1-\cdots-s_{k-1}),\qquad 
t\in \widehat{G}, \ s\in \widehat{G}^{\,\mathbb Z}.
$$
Applying (\ref{E}), we deduce that for $t\in\widehat{G}$, we have
$$
[\Edb_J U^k J(f)](t)
=\int_{\widehat{G}^{\{0,\ldots,k-1\}}} 
f(t-s_0-s_1-\cdots -s_{k-1})\, d\mu(s_{k-1})\cdots d\mu(s_0).
$$
Hence we obtain that
$$
\Edb_JU^kJ(f) = T^k(f).
$$
\end{proof}

Our next goal is to extend part (ii) of 
Theorem \ref{Equiv} to compact groups.

\begin{lemma}\label{Trace}
Let \( G \) be a compact group. 
The canonical embedding of 
\( VN(G) \) into 
\( B(L^2(G)) \) is trace 
preserving.
\end{lemma}

\begin{proof}
We let $du$ denote the Haar measure on $G$.
Let \(f \in L^2(G)\). Using the fact that $L^2(G)\subseteq L^1(G)$,
we may write, for all \(g \in L^2(G)\),
\[
[\lambda(f)g](t) = \int_G f(tu^{-1})g(u) \, du.
\]
Hence, 
\([\lambda(f) g](t) = \int_G F(u, t) g(u) \, du\), where \(F(u, t) := f(tu^{-1})\). We observe that 
$$
\|F\|_{L^2(G \times G)} = \|f\|_{L^2(G)},
$$
because the Haar measure on a compact group is a probability measure. Applying e.g.  \cite[Theorem VI. 23]{RS}, this implies that \(\lambda(f) \in S^2(L^2(G))\), with \(\|\lambda(f)\|_{S^2} = \|f\|_{L^2}\).
Thus,
\[
\text{tr}(\lambda(f)^* \lambda(f)) 
= \|f\|_{L^2}^2.
\]
Applying (\ref{Plancherel}), we deduce that 
$$
\text{tr}(\lambda(f)^* \lambda(f))
=\tau_G(\lambda(f)^* \lambda(f)),
\quad f\in L^2(G).
$$
We deduce that \(\tau_G\) is the restriction of \(\text{tr}\) 
through the embedding \(VN(G) \hookrightarrow B(L^2(G))\).
\end{proof}

\begin{proposition}\label{cpgp}
Let \( G \) be a compact group. 
Let $u: G\to \Cdb$ be a unital, positive definite continuous function.
If the Schur multiplier \( T_{\phi_u} \) has an absolute dilation, then the Fourier multiplier \( M_u \) also admits an absolute dilation.
\end{proposition}

\begin{proof}
Note that $M_u$ is the restriction of $T_{\phi_u}$ to $VN(G)$.
Hence using Lemma \ref{Trace}, it suffices to compose 
the absolute dilation
of $T_{\phi_u}$ by the embedding $VN(G)\hookrightarrow B(L^2(G))$
to obtain an absolute dilation of $M_u$.
\end{proof}

\section{Counterexample: A ucp
Fourier multiplier without absolute dilation}\label{Sec4}
In this section, we construct a ucp Fourier multiplier that does not admit an absolute dilation. We focus on the symmetric group \( \S_3 \). Our goal is to define a function \( u\colon \S_3\to\mathbb{C} \) such that the map \( M_u\colon VN(\S_3)\to VN(\S_3) \) is ucp, while the associated Herz-Schur multiplier is not factorizable. Applying Lemma \ref{Equiv} and Theorem \ref{Transf}~(i), we conclude that \( M_u \) has no absolute dilation.

Our approach relies on the criterion established 
in \cite[Corollary 2.3]{HM} recalled below.

\begin{proposition}\label{Prop 5.1}
Let $n\geq 1$ be an integer, let us equip
$M_n(\Cdb)$ with 
the normalized trace $\tau_n$ and let 
\(T:M_n(\mathbb{C})\to M_n(\mathbb{C})\) be a ucp
trace preserving map of the form
\[
T(x)=\sum_{k=1}^d a_k^* x a_k,\quad x\in M_n(\mathbb{C}),
\]
where \(a_1,\ldots, a_d\in M_n(\mathbb{C})\). 
If \(d\geq2\) and the family 
\[
\{a_k^*a_l:1\leq k,l\leq d\}
\]
is linearly independent, then \(T\) is not factorizable.
\end{proposition}
\bigskip

For any 
$$
\varphi=
\begin{pmatrix}
\varphi_1\\ \vdots\\ \varphi_n
\end{pmatrix}\quad\hbox{and}\quad
\varphi'=
\begin{pmatrix}
\varphi'_1\\ \vdots\\ \varphi'_n
\end{pmatrix}
$$
in $\mathbb{C}^n$, we set 
$$
\varphi'\otimes\varphi = 
\bigl[\varphi'_i\varphi_j]_{1\leq i,j\leq n}
\in M_n(\mathbb{C}).
$$
Also we set
$$
\overline{\varphi}=
\begin{pmatrix}
\overline{\varphi_1}\\ \vdots\\ \overline{\varphi_n}
\end{pmatrix}
\quad\hbox{and}\qquad
\varphi'\cdotp\varphi = 
\begin{pmatrix}
\varphi'_1\varphi_1\\ \vdots\\\varphi'_1 \varphi_n
\end{pmatrix}.
$$
Note that $\overline{\varphi}\otimes\varphi$ is a positive semidefinite matrix.

Using the argument in \cite[Example 3.2]{HM}, we 
obtain the following
criterion concerning Schur multipliers.

\begin{corollary}\label{Cor. 5.2}
Let $n\geq 1$ be an integer, let $d\geq 2$, let
$\varphi(1),\ldots,\varphi(d)\in\mathbb{C}^n$ and let 
$$
A=\sum_{k=1}^d \overline{\varphi(k)}\otimes \varphi(k)\,\in M_n(\mathbb C).
$$
Assume that the diagonal entries of $A$ are equal to $1$ 
and that 
the family 
\begin{equation}\label{LinInd}
\{\overline{\varphi(k)}\cdotp\varphi(l)  :1\leq k,l\leq d\}
\end{equation}
is linearly independent. Then the Schur multiplier $T_A : M_n(\Cdb)\to M_n(\Cdb)$
associated
with $A$ is ucp but it is not factorizable.
\end{corollary}

\begin{proof}
The Schur multiplier $T_A$ is ucp by Lemma \ref{Positive}. For any $\varphi={}^t \!(\varphi_1,\ldots,\varphi_n)\in\mathbb{C}^n$,
let ${\mathfrak a}(\varphi)\in M_n(\mathbb{C})$ 
be the diagonal matrix with entries
$\varphi_1,\ldots,\varphi_n$. Then $x\mapsto 
{\mathfrak a}(\varphi)^*x{\mathfrak a}(\varphi)$ 
is the Schur multiplier associated with 
$\overline{\varphi}\otimes\varphi$. 
Hence  we have
\[
T_A(x)=\sum_{k=1}^d {\mathfrak a}(\varphi(k))^* x {\mathfrak a}(\varphi(k)),\quad x\in M_n(\mathbb{C}).
\]
For any $1\leq k,l\leq n$, we have
${\mathfrak a}(\varphi(k))^*{\mathfrak a}(\varphi(l))=
{\mathfrak a}(\overline{\varphi(k)}\cdotp\varphi(l))$.
Hence 
the family (\ref{LinInd}) is linearly independent in
$\mathbb{C}^n$ if and only if the family 
$\{{\mathfrak a}(\varphi(k))^*{\mathfrak a}(\varphi(l))
:1\leq k,l\leq d\}$
is linearly independent in
$M_n(\mathbb{C})$. 
The result therefore follows from Proposition 
\ref{Prop 5.1}.
\end{proof}

\begin{theorem}\label{CEX}
There exists a function $u\colon\S_3\to\Cdb$ such that $M_u\colon VN(\S_3)\to VN(\S_3)$ is ucp but $M_u$
does not admit any absolute dilation.
\end{theorem}

\begin{proof}
We identify each \( \sigma \in \S_3 \) with the basis vector \( \epsilon_\sigma \in \ell^2_{{\mathcal S}_3} \). We enumerate the six elements of \( \S_3 \) as follows  
\[
\mathcal{B} = \{1, (123), (132), (12), (23), (31)\}.
\]
Let \( (a, b, c, d, e, f) \in \mathbb{C}^6 \). 
Let $V\in VN(G)$ be defined by
\begin{equation}\label{V}
V = a\lambda(1) + b\lambda(123) + c\lambda(132) + d\lambda(12) + e\lambda(23) + f\lambda(31).
\end{equation}
Let \( A \) be the matrix of $V$ in the standard basis 
of \( \ell^2_{{\mathcal S}_3} \),
according to the enumeration \( \mathcal{B} \). 
It can be easily verified that the matrix \( A \) is given by
\begin{equation}\label{A0}
A = \begin{pmatrix} 
a & c & b & d & e & f \\
b & a & c & f & d & e \\
c & b & a & e & f & d \\
d & f & e & a & b & c \\
e & d & f & c & a & b \\
f & e & d & b & c & a
\end{pmatrix}.
\end{equation}

Let \( u \colon \S_3 \to \mathbb{C} \) be defined by
\[
u(1) = a, \ u(123) = b, \ u(132) = c, \ u(12) = d, \ u(23) = e, \ u(31) = f.
\]
Then one can check that 
the matrix \( [u(st^{-1})]_{(s,t) \in {\mathcal S}_3^2} \) written according 
to the enumeration \( \mathcal{B} \) 
is also equal to \( A \).

Hence by Corollary \ref{Cor. 5.2}, Lemma \ref{HS} and the discussion preceding Proposition \ref{Prop 5.1}, it suffices to 
find a matrix $A$ of the form (\ref{A0}), with $a=1$, as well
as two vectors $\varphi,\psi\in{\mathbb C}^6$ such that 
\begin{equation}\label{Rk2}
A = \overline{\varphi} \otimes \varphi + \overline{\psi} \otimes \psi,
\end{equation}
and such that the quadruple 
\begin{equation}\label{Quad}
\{ \overline{\varphi}\cdotp \varphi, 
\overline{\varphi}\cdotp \psi, \overline{\psi} \cdotp\varphi, \overline{\psi}\cdotp \psi \}
\end{equation}
is linearly independent in \( \mathbb{C}^6 \). 

Instead of introducing \(\varphi\) and \(\psi\) without context, we will first explain how these two vectors are constructed. While this approach slightly lengthens the proof, it aims to make it more informative.

To proceed, we 
apply the Peter-Weyl theorem (see e.g. \cite[Section 5.2]{F}) to the group \( \S_3 \). Let
\( I \colon \S_3 \to \{1\} \subset \mathbb{C} \) be the trivial representation and let \( \varepsilon \colon \S_3 \to \{-1, 1\} \subset \mathbb{C} \) be the signature. Next, let 
$$
E = \{(x_1, x_2, x_3) \in \mathbb{C}^3 : x_1 + x_2 + x_3 = 0 \},
$$
equipped with the Hermitian structure inherited from $\mathbb{C}^3$,
and let \( \pi \colon \S_3 \to B(E) \) be the representation
defined by
\[
\pi(\sigma) \begin{pmatrix} x_1 \\ x_2 \\ x_3 \end{pmatrix} = \begin{pmatrix} x_{\sigma(1)} \\ x_{\sigma(2)} \\ x_{\sigma(3)} \end{pmatrix}, \quad \sigma \in S_3.
\]
Then up to unitary equivalence,
\( I \), \( \varepsilon \), and \( \pi \) are the three (non-equivalent) 
irreducible representations of \( \S_3 \).
Hence  the Peter-Weyl theorem ensures the existence of a 
unitary operator \( U \colon \ell^2_{{\mathcal S}_3} \to \mathbb{C} 
\oplus \mathbb{C} \oplus E \oplus E \)
such that for all \(\sigma \in \S_3\),
\begin{equation}\label{PW}
\lambda(\sigma) = U^* \left[ I(\sigma) \oplus \varepsilon(\sigma) \oplus \pi(\sigma) \oplus \pi(\sigma) \right] U.
\end{equation}
Choose the orthonormal basis \( (\xi, \eta) \) 
of \( E \) given by
\[
\xi = \frac{1}{\sqrt{3}} \begin{pmatrix} 1 \\ j \\ j^2 \end{pmatrix}, \quad \eta = \frac{1}{\sqrt{3}} \begin{pmatrix} 1 \\ j^2 \\ j \end{pmatrix},
\]
where \( j \) is a primitive cube root of unity. We easily check that in this basis, we have
\[
\pi(123) = \begin{pmatrix} j & 0 \\ 0 & j^2 \end{pmatrix}, \quad \pi(132) = \begin{pmatrix} j^2 & 0 \\ 0 & j \end{pmatrix},
\]
\[
\pi(12) = \begin{pmatrix} 0 & j^2 \\ j & 0 \end{pmatrix}, \quad \pi(23) = \begin{pmatrix} 0 & 1 \\ 1 & 0 \end{pmatrix}, \quad \pi(31) = \begin{pmatrix} 0 & j \\ j^2 & 0 \end{pmatrix}.
\]
Let
\[
\Gamma = a\pi(1) + b\pi(123) + c\pi(132) + d\pi(12) + e\pi(23) + f\pi(31).
\]
Let \( B \) be the matrix of \( \Gamma \) in the basis \( (\xi, \eta) \). It follows from above that
\begin{equation}
B = \begin{pmatrix} a + bj + cj^2 & dj^2 + e + fj \\ dj + e + fj^2 & a + bj^2 + cj \end{pmatrix}.
\end{equation}
In the sequel, we use the symbol $\sim$ for unitary equivalence of matrices. Since $A$ is a matrix of the operator $V$ 
given by (\ref{V}) 
in an orthonormal basis, 
(\ref{PW}) provides 
\begin{equation}\label{unieqi}
A \sim 
[a + b + c + d + e + f] \oplus 
[a + b + c - d - e - f] \oplus B \oplus B,
\end{equation}
where the matrix
on the right-hand side is written 
in a block-diagonal form 
(with two $1\times 1$ blocks and two $2\times 2$ blocks).

We look for \( (a, b, c, d, e, f) \) with \( a = 1 \), 
such that \( A \geq 0 \), and \( \text{rk}(A) = 2 \), so that
$A$ will have the form (\ref{Rk2}).
To achieve this, it suffices to have
\[
a + b + c = 0, \quad d + e + f = 0, \quad B\geq 0\quad\quad \text{and} \quad \text{rk}(B) = 1.
\]
It turns out that there exists \( (a, b, c, d, e, f) \) 
satisfying these conditions, with
\[
B = \begin{pmatrix} 1 & \sqrt{2} \\ \sqrt{2} & 2 \end{pmatrix}.
\]
Finding $(a,b,c,d,e,f)$ as above amounts to have
\[
\left\{\begin{aligned}
a&=1\\
a+b+c&=0\\
a+bj+cj^2&=1\\
a+bj^2+cj&=2
\end{aligned}\right.
\qquad \text{and} \qquad
\left\{\begin{aligned}
d+e+f&=0\\
dj^2+ e +fj&=\sqrt2\\
dj+ e +fj^2&=\sqrt2
\end{aligned}\right.
\]
These two systems actually have a unique solution, given by  
\[
a = 1, \quad b = -\frac{1}{2} + \frac{i}{2\sqrt{3}}, \quad c = \overline{b} = -\frac{1}{2} - \frac{i}{2\sqrt{3}}, \quad d = f = -\frac{\sqrt{2}}{3}, \quad e = \frac{2\sqrt{2}}{3}.
\]
From now on, $A$ will be the matrix defined with these 
values of $a,b,c,d,e,f$.
Thus,
\begin{equation}\label{A}  
A = \begin{pmatrix}  
1 & \overline{b} & b & -\delta & 2\delta  & -\delta \\  
b & 1 & \overline{b} & -\delta  & -\delta  & 2\delta \\  
\overline{b} & b & 1 & 2\delta & -\delta  & -\delta \\  
-\delta & -\delta & 2\delta & 1 & b & \overline{b} \\  
2\delta & -\delta  & -\delta  & \overline{b} & 1 & b \\  
-\delta  & 2\delta & -\delta  & b & \overline{b} & 1  
\end{pmatrix},  
\end{equation}  
where  
\[
\delta = \frac{\sqrt{2}}{3} \quad\hbox{and}\quad  
b = -\frac{1}{2} + \frac{i}{2\sqrt{3}}.
\]  
We observe that
$$
B\sim\begin{pmatrix} 0 & 0 \\ 0 & 3\end{pmatrix},
$$
which implies that  
\begin{equation}\label{Diag}  
A \sim {\rm Diag}\{0,0,0,0,3,3\}.  
\end{equation} 
We will use this to write $A$ in the form (\ref{Rk2}).

We set
\[
\varphi = \begin{pmatrix} 1 \\ b \\ \overline{b} \\ -\delta \\ 2\delta \\ -\delta \end{pmatrix},
\qquad
\varphi' = \begin{pmatrix} \overline{b} \\ 1 \\ b \\ -\delta \\ -\delta \\ 2\delta \end{pmatrix}\qquad\hbox{and}\qquad
\psi := \varphi' - \frac{\langle \varphi', \varphi \rangle}{\|\varphi\|^2} \, \varphi.
\]
The first row of $A$ (regarded as a functional on $\Cdb^6$) 
is equal to
$\langle\,\cdotp,\varphi\rangle$. Hence if $X\in{\rm Ker}(A)$,
we have
$\langle X,\varphi\rangle=0$. By (\ref{Diag}), this implies that 
\begin{equation}\label{3}
A\varphi=3\varphi.
\end{equation}
Likewise (using the second row of $A$), we obtain that $A\varphi'=3\varphi'$.
Hence $(\varphi,\psi)$ is an orthogonal basis of ${\rm Ker}(A-3 I)$. We now compute $\psi$. 
Taking the first coordinates
in the relation (\ref{3}), we 
obtain that $\norm{\varphi}^2=3$. Next,
\begin{align*}
\langle \varphi',\varphi\rangle &= \overline{b} + \overline{b} + b^2 +\delta^2 -2\delta^2 -2\delta^3\\
& = 2\overline{b} + b^2 -3\delta^2.
\end{align*}
Moreover 
\[
b^2 = \frac{1}{6} - \frac{i}{2\sqrt{3}} 
\quad \text{and} \quad  
\delta^2 = \frac{2}{9}.
\]
Hence 
$$
b^2 -3\delta^2 =\frac16 -\frac23 -\frac{i}{2\sqrt3} = -\frac12-\frac{i}{2\sqrt3} =\overline{b}.
$$
This yields
$$
\langle \varphi',\varphi\rangle=3\overline{b}
\quad\hbox{and}\quad
\psi =\varphi'-\overline{b}\varphi.
$$
Consequently,
\[
\psi = \begin{pmatrix}  
0 \\ 1 - \vert b \vert^2 \\ b - \overline{b}^2 \\  
\delta (\overline{b} - 1) \\ -\delta (2\overline{b} + 1) \\ \delta (\overline{b} + 2)  
\end{pmatrix}.  
\]
Moreover,
$$
\norm{\psi}^2= \norm{\varphi'}^2 +\vert b\vert^2\norm{\varphi}^2
-2{\rm Re}\bigl(b\langle\varphi',\varphi\rangle\bigr).
$$
We note that $\norm{\varphi'}^2=1$ (same argument as $\varphi$) and
that 
\begin{equation}\label{b}
\vert b\vert^2=\frac13.
\end{equation}
We deduce that 
$\norm{\psi}^2=2$.

By (\ref{Diag}), $A$ is equal to 3 times 
the orthogonal projection on ${\rm Ker}(A-3 I)$.
It therefore follows from above that
$$
A = \overline{\varphi}\otimes \varphi + \frac32\,\overline{\psi}\otimes\psi.
$$

It now suffices to show that the quadruple (\ref{Quad}) is linearly independent. 
Using (\ref{b}), we compute
$$
\overline{\varphi}\cdotp\varphi= \begin{pmatrix} 1 \\ \frac13 \\ \frac13 \\ \frac29 \\  \frac89 \\ \frac29\end{pmatrix},
\quad
\overline{\psi}\cdotp\psi=\begin{pmatrix} 0 \\ \frac49 \\ \frac49 \\ \frac{14}{27} \\  \frac{2}{27} \\ \frac{14}{27}\end{pmatrix},
\quad
\overline{\psi}\cdotp\varphi=\begin{pmatrix} 0 \\ \frac23 b\\ \overline{b}^2 -\frac{b}{3} \\ 
-\frac29(b-1) \\  -\frac49(2b+1) \\ -\frac29(b+2)\end{pmatrix},
\quad
\overline{\varphi}\cdotp\psi=\begin{pmatrix} 0 \\ \frac23 \overline{b}\\ b^2 -\frac{\overline{b}}{3} \\ 
-\frac29(\overline{b}-1) \\  -\frac49(2\overline{b}+1) \\ -\frac29(\overline{b}+2)\end{pmatrix}.
$$
The first entries of these four vectors are $1,0,0,0$, hence it suffices to show that the last three ones,
$\{\overline{\psi}\cdotp\psi, \overline{\psi}\cdotp\varphi, \overline{\varphi}\cdotp\psi\}$
are linearly independent. We consider the matrix $M$ of these three vectors,
$$
M = \begin{pmatrix} 0 & 0 & 0 \\
\frac49 &\frac23 b & \frac23 \overline{b} \\
\frac49 & \overline{b}^2 -\frac{b}{3} & b^2-\frac{\overline{b}}{3} \\
\frac{14}{27} & -\frac29(b-1) & -\frac29(\overline{b}-1)\\
\frac{2}{27} &  -\frac49(2b+1) &  -\frac49(2\overline{b}+1)\\
\frac{14}{27} & -\frac29(b+2)& -\frac29(\overline{b}+2)
\end{pmatrix}.
$$
Let $\Delta$ be the determinant of the rows 2,3,5 of $M$, that is
$$
\Delta ={\rm det}\begin{pmatrix} 
\frac49 &\frac23 b & \frac23 \overline{b} \\
\frac49 & \overline{b}^2 -\frac{b}{3} & b^2-\frac{\overline{b}}{3} \\
\frac{2}{27} &  -\frac49(2b+1) &  -\frac49(2\overline{b}+1)
\end{pmatrix}.
$$
Then $\Delta= \frac{8\sqrt{3} i}{81}
\not=0$. Hence ${\rm rk}(M)=3$ and we are done.

\end{proof}

\begin{remark}
\ 

\smallskip (1)
\cite[Example 3.2]{HM} provides an example of a non-factorizable ucp Schur multiplier, which is not a Herz-Schur multiplier. In contrast, our example above is a Herz-Schur multiplier. Moreover, it represents the smallest possible instance of such a multiplier, since any group of smaller order than $\S_3$ must be abelian, and \cite[Remark 6.3(a)]{AD} ensures 
that every ucp Herz-Schur multiplier in that setting is factorizable.
(see also Theorem 
\ref{abgp}).

\smallskip (2) As a von Neumann algebra, 
$VN(\S_3)$ is isomorphic to 
$\ell^\infty_2\oplus M_2(\Cdb)$. Hence Theorem \ref{CEX} implicitly provides a 
trace preserving ucp map
on $\ell^\infty_2\oplus M_2(\Cdb)$ which is not factorizable. This contrasts with 
the fact that any trace preserving ucp map
on $M_2(\Cdb)$ is factorizable
(see \cite{K}).
\end{remark}

\begin{remark}
Let \( (\M, \tau) \) be a tracial von Neumann algebra. Let \( 1 < p < \infty \) and 
let \( S : L^p(\M, \tau) \to L^p(\M, \tau) \) be a completely positive contraction. We say that \( S \) admits a completely isometric dilation if there exist another tracial von Neumann algebra \( (M, \sigma) \), two complete contractions \( J : L^p(\M, \tau) \to L^p(M, \sigma) \) and \( Q : L^p(M, \sigma) \to L^p(\M, \tau) \), as well as a complete surjective isometry \( V : L^p(M, \sigma) \to L^p(M, \sigma) \), such that \( S^k = QV^kJ \) for all \( k \geq 0 \). For more information on this topic, we refer to \cite{A, D0, D, DL1, JL} and the references therein. The first example of a completely positive contraction on a noncommutative \( L^p \)-space that does not admit a completely isometric dilation was provided in \cite{JL}. Building on \cite[Example 3.2]{HM}, C. Duquet demonstrated in \cite{D} the existence of a ucp Schur multiplier \( S \) on \( M_4(\mathbb{C}) \) such that, for all \( 1 < p \neq 2 < \infty \), \( S : S^p_4 \to S^p_4 \) does not admit a completely isometric dilation.

Thanks to Theorem \ref{CEX}, we can now obtain a ucp Fourier multiplier \( S : VN(\S_3) \to VN(\S_3) \) such that for all \( 1 < p \neq 2 < \infty \), \( S : L^p(VN(\S_3)) \to L^p(VN(\S_3)) \) does not admit a completely isometric dilation. Indeed, let \( S = M_u \) as given in Theorem \ref{CEX}, and let \( 1 < p \neq 2 < \infty \). Assume for the sake of contradiction that \( S : L^p(VN(\S_3)) \to L^p(VN(\S_3)) \) has a completely isometric dilation. Then, arguing as in the proof of part (i) of Theorem \ref{Transf}, we obtain that the Herz-Schur multiplier \( T_{\phi_u} : S^p(\ell^2_{{\mathcal S}_3}) \to S^p(\ell^2_{{\mathcal S}_3}) \) also admits a completely isometric dilation. Applying \cite[Theorem 2.10]{D}, this implies that \( T_{\phi_u} : B(\ell^2_{{\mathcal S}_3}) \to B(\ell^2_{{\mathcal S}_3}) \) admits an absolute dilation, leading to a contradiction.

This yields a completely positive contraction without any completely isometric dilation on a noncommutative \( L^p \)-space of dimension
6.
\end{remark}

\bigskip
\noindent
{\bf Acknowledgements.} 
We express our gratitude to \'Eric Ricard for his stimulating and insightful discussions. We also extend our 
thanks to the International Center for Mathematical Sciences (ICMS) for fostering an exceptional research
 environment, which greatly accelerated and facilitated our work through a RiP grant. Additionally, 
the second author gratefully acknowledges the University of Marie and Louis Pasteur for providing an excellent research environment during her visit.

\bigskip

\bigskip

\begin{thebibliography}{99}
\bibitem{AD} C. Anantharaman-Delaroche, {\it On ergodic theorems for free group actions on noncommutative spaces}, Probab. Theory Relat. Fields 135 (4) (2006) 520–546.
\bibitem{A} C. Arhancet, {\it On Matsaev's conjecture for contractions 
on noncommutative $L_p$-spaces}, 
J. Operator Theory 69 (2013), no. 2, 387–421. 
\bibitem{AKLZ} C. Arhancet, C. Kriegler, C. Le Merdy, S. Zadeh,
{\it Separating Fourier and Schur multipliers}, 
J. Fourier Anal. Appl. 30 (2024), no. 1, Paper No. 5, 27 pp.
\bibitem{BF} M. Bożejko, G. Fendler, 
{\it Herz-Schur multipliers and completely bounded multipliers of the Fourier algebra 
of a locally compact group}, Boll. Un. Mat. Ital. A (6) 3 (1984), no. 2, 297–302. 
\bibitem{BF2} M. Bożejko, G. Fendler, 
{\it Herz-Schur multipliers and uniformly bounded representations of discrete groups},
Arch. Math. 57 (1991), no. 3, 290–298. 
\bibitem{DCH} J. De Canni\`ere and U. Haagerup, 
{\it Multipliers of the Fourier algebras of some simple 
Lie groups and their discrete subgroups}, 
Amer. J. Math. 107 (1985), 455--500.
\bibitem{D0} C. Duquet,
{\it Dilation properties of measurable Schur multipliers and Fourier multipliers}, Positivity 26
(2022), no. 4, Paper No. 69, 41 pp.
\bibitem{D} C. Duquet, 
{\it Unital positive Schur multipliers on $S^p_n$ 
with a completely isometric dilation},
Math. Scand. 130 (2024), no. 3, 505–525. 
\bibitem{DL1}  C. Duquet and C. Le Merdy, {\it A characterization of absolutely dilatable Schur multipliers}, Adv. Math. 439 (2024), Paper No. 109492, 33 pp.
\bibitem{DLM2} C. Duquet and C. Le Merdy, {\it Absolute dilations of UCP self-adjoint Fourier multipliers: the nonunimodular case}, to appear in Glasgow Mathematical Journal. arXiv:2406.06074.
\bibitem{F} G. Folland, 
{\it A course in Abstract Harmonic Analysis}, Studies in Advanced Mathematics,  CRC Press, Boca Raton, FL, 1995. x+276 pp.
\bibitem{Haag} U. Haagerup, {\it On the dual weights for crossed products of von Neumann algebras II}, Math. Scand. 43 (1978) 119–140.
\bibitem{HJX} U. Haagerup, M. Junge, Q. Xu, {\it
A reduction method for noncommutative $L_p$-spaces and applications}, 
Trans. Amer. Math. Soc. 362 (2010), no. 4, 2125–2165. 
\bibitem{HM} U. Haagerup, M. Musat, {\it Factorization and dilation problems for completely positive maps on von Neumann algebras}, Commun. Math. Phys. 303 (2) (2011) 555–594.
\bibitem{J} P. Jolissaint, {\it A characterization of 
completely bounded multipliers of Fourier algebras}, Colloq. Math. 63 (1992), no. 2, 311–313.
\bibitem{K} B. Kümmerer, {\it Markov dilations on the 2×2 matrices. In: Operator algebras and their connections with topology and ergodic theory} (Busteni, 1983), Lect. Notes Math. 1132, Berlin: Springer, 1985, pp. 312–323.
\bibitem{k1} B. Kümmerer, {\it Markov dilations on {$W^\ast$}-algebras}, J. Funct. Anal., 63 (1985) no. 2, 139-177.
\bibitem{McD} D. McDuff, {\it Uncountably many $II_1$ factors}, 
Ann. of Math. (2) 90 (1969), 372–377.
\bibitem{JL} M. Junge, C. Le Merdy, 
{\it Dilations and rigid factorisations on noncommutative 
$L_p$-spaces}, J. Funct. Anal. 249 (2007), no. 1, 220–252. 
\bibitem{NT} M. Nagisa, J. Tomiyama, {\it
Completely positive maps in the tensor products of von Neumann algebras}, 
J. Math. Soc. Japan 33 (1981), no. 3, 539–550. 
\bibitem{NR} S. Neuwirth, \'E. Ricard, 
{\it Transfer of Fourier multipliers into Schur multipliers and sumsets in a discrete group},
Canad. J. Math. 63 (2011), no. 5, 1161–1187. 
\bibitem{P} V. Paulsen, {\it Completely bounded maps and operator algebras}, Cambridge Studies in Advanced Mathematics 78, Cambridge University Press, Cambridge, 2002
\bibitem{Pi} G. Pisier, {\it Tensor products of 
$C^*$-algebras and operator spaces—the Connes-Kirchberg problem},
London Mathematical Society Student Texts, 96. Cambridge University Press, Cambridge, 2020. x+484 pp.
\bibitem{PX} G. Pisier, Q. Xu,
{\it Non-commutative $L_p$-spaces},
Handbook of the geometry of Banach spaces, Vol. 2, 1459–1517, North-Holland, Amsterdam, 2003.
\bibitem{RS} M. Reed, B. Simon, 
{\it Methods of modern mathematical physics. I. Functional analysis, Second edition},
Academic Press, New York, 1980, xv+400 pp.
\bibitem{R} \'E. Ricard, {\it A Markov dilation for self-adjoint Schur multipliers},  Proc. Amer. Math. Soc. 136 (2008), no. 12, 4365–4372. 
\bibitem{Tak} M. Takesaki, {\it Theory of operator algebras I}, Springer-Verlag, New York-Heidelberg, 1979, vii+415 pp.
\bibitem{Tak2}  M. Takesaki, {\it Theory of operator algebras. II}, Encyclopaedia of Mathematical Sciences, 125, 
Operator Algebras and Non-commutative Geometry, 6, Springer-Verlag, Berlin, 2003, xxii+518 pp.
\end{thebibliography}
\end{document}